\documentclass[reqno]{amsart}

\usepackage{amsmath, amssymb, amsthm, epsfig}
\usepackage{hyperref, latexsym}
\usepackage{url}
\usepackage[mathscr]{euscript}
\usepackage[latin1]{inputenc}
\usepackage{color}
\usepackage{harpoon}
\usepackage{url}

\usepackage{setspace}
\usepackage{refcheck} 
\norefnames
\nocitenames

\def\today{\ifcase\month\or
  January\or February\or March\or April\or May\or June\or
  July\or August\or September\or October\or November\or December\fi
  \space\number\day, \number\year}

\newtheorem{theorem}{Theorem}
\newtheorem{conjecture}{Conjecture}
\newtheorem{lemma}[theorem]{Lemma}
\newtheorem{corollary}[theorem]{Corollary}
\theoremstyle{definition}

\theoremstyle{remark}

\newcommand{\ft}{\widehat}
\newcommand{\mc}{\mathcal}

\newcommand{\B}{\mc{B}}

\newcommand{\F}{\mc{F}}

\newcommand{\G}{\mc{G}}

\newcommand{\K}{\mc{K}}
\newcommand{\LL}{\mc{L}}
\newcommand{\M}{\mc{M}}
\newcommand{\OO}{\mc{O}}

\newcommand{\PP}{\mc{P}}
\newcommand{\Sp}{\mathbb{S}}

\newcommand{\W}{\mc{W}}
\newcommand{\C}{\mathbb{C}}

\newcommand{\R}{\mathbb{R}}

\newcommand{\Q}{\mc{Q}}
\newcommand{\Z}{\mathbb{Z}}

\newcommand{\bx}{\boldsymbol{x}}
\newcommand{\by}{\boldsymbol{y}}
\newcommand{\bz}{\boldsymbol{z}}
\newcommand{\bxi}{\boldsymbol{\xi}}
\newcommand{\bze}{\boldsymbol{\zeta}}
\newcommand{\bn}{{\boldsymbol{n}}}
\newcommand{\bv}{{\boldsymbol{v}}}
\newcommand{\bu}{{\boldsymbol{u}}}
\newcommand{\bm}{{\boldsymbol{m}}}
\newcommand{\bl}{{\boldsymbol{l}}}
\newcommand{\p}{\varphi}
\newcommand{\ds}{\text{\rm d}s}

\newcommand{\dt}{\text{\rm d}t}
\newcommand{\du}{\text{\rm d}u}

\newcommand{\dx}{\text{\rm d}x}

\renewcommand{\d}{\text{\rm d}}

\newcommand{\es}[1]{\begin{equation}\begin{split}#1\end{split}\end{equation}}
\newcommand{\est}[1]{\begin{equation*}\begin{split}#1\end{split}\end{equation*}}

\newcommand{\ov}{\overline}
\renewcommand{\H}{\mc{H}}

\newcommand{\re}{{\rm Re}\,}

\newcommand{\bo}{\boldsymbol}

\newcommand{\om}{\omega}
\newcommand{\la}{\lambda}
\newcommand{\ga}{\gamma}

\begin{document}


\title[Gon\c{c}alves]{Orthogonal Polynomials and Sharp Estimates for the Schr\"odinger Equation}
\author[Gon\c{c}alves]{Felipe Gon\c{c}alves}
\date{\today}
\address{University of Alberta, Mathematical and Statistical Sciences, CAB 632, Edmonton, Alberta, Canada T6G 2G1}
\email{felipe.goncalves@ualberta.ca}
\urladdr{sites.ualberta.ca/$\sim$goncalve}
\allowdisplaybreaks
\numberwithin{equation}{section}


\begin{abstract}
In this paper we study sharp estimates for the Schr\"odinger operator via the framework of orthogonal polynomials. We use Hermite and Laguerre polynomial expansions to produce sharp Strichartz estimates for even exponents. In particular, for radial initial data in dimension $2$, we establish an interesting connection of the Strichartz norm with a combinatorial problem about words with four letters. We use spherical harmonics and Gegenbauer polynomials to prove a sharpened weighted inequality for the Schr\"odinger equation that is maximized by radial functions.
\end{abstract}


\maketitle


\section{Introduction}
Let $2\leq p,q \leq \infty$. The Strichartz estimate for the Schr\"odinger equation (see \cite[Theorem 2.3]{Tao}) states that there exists a constant $C$ such that
\es{\label{St-est}
\|\|e^{it\Delta}f(\bx)\|_{L^p(\d\bx)}\|_{L^q(\d t)} \leq C \|f(\bx)\|_{L^2(\d\bx)},
}
for all $f\in L^2(\R^d,\d\bx)$, where $\Delta$ is the Laplacian in $\R^d$ and $e^{it\Delta}f(\bx)$ denotes the solution of the Schr\"odinger equation $\partial_t u(\bx,t)=i\Delta u(\bx,t)$ with initial data $u(\bx,0)=f(\bx)$. The exponents above satisfy the following relation
\est{
\frac{d}{p}+\frac{2}{q} = \frac{d}{2}, \ \ \ (p,q,d)\neq (\infty,2,2).
}
Since Strichartz's original work \cite{St} in 1977, the search for maximizers of this space-time estimate was launched and it is conjectured that a function $f(\bx)$ maximizes the above inequality if and only if $f(\bx)$ is a Gaussian.
 
\begin{conjecture}\label{St-conj}
A function $f(\bx)$ maximizes \eqref{St-est} if and only if it has the form $Ae^{-B\|\bx\|^2 + \bu\cdot\bx}$, where $A,B\in\C$, $\re B>0$ and $\bu\in\C^d$.
\end{conjecture}

\noindent {\bf Remark.} We note that Gaussian functions maximize \eqref{St-est} if and only if it holds with $C=C(p,d)$ given by
\es{\label{sharp-const}
C(p,d)=(p^{-1/2p}2^{1/p-1/4})^d.
}

\noindent The first to prove this conjecture for $(p,q,d)\in\{(6,6,1);(4,4,2)\}$ was Foschi \cite{Fo} in 2004. In 2005, Hundertmark and Zharnitsky \cite{HZ} gave an alternative proof for these two cases. Later on, in 2008, Carneiro \cite{Ca} and Bennett, Bez, Carbery and Hundertmark \cite{BBCH} gave alternative proofs for these cases, including in addition the new case $(p,q,d)=(4,8,1)$. All these proofs, although shedding new light at the problem via different angles, heavily rely on the crucial fact that $p=2k$ and $q=2k\ell$ for some integers $k\geq 2$ and $\ell\geq 1$. This vital property opens several doors to approach the problem (and we intend to open another one with this article), but the general case for non-even exponents still remains unsolved, although maximizers are known to exist (see \cite{Sh}).
\smallskip

The main goal of this paper is to demonstrate how orthogonal polynomials can be used to prove sharp space-time estimates related to the Schr\"odinger operator. The novelty of the present work lies in the use of techniques associated with the theory of orthogonal polynomials (in the sense of \cite[Chapter 2]{Sz}) to attack these problems and which ultimately allows us to: $(i)$ Prove a sharpened weighted inequality for the Schr\"odinger operator (Theorem \ref{radial-thm} and Corollary \ref{radial-cor}); $(ii)$ Develop a new way of attacking Conjecture \eqref{St-conj} (Theorems \ref{equiv-thm-hermite} and \ref{equiv-thm-laguerre}); $(iii)$ Produce alternative proofs for the fact that Gaussians maximize the Strichartz estimate \eqref{St-est} in the case of even exponents (Theorems \ref{L2-L2-estimate-hermite} and \ref{QS-op-thm} and Corollaries \ref{cor} and \ref{radial-St-sharp}). Moreover, for the case $(p,q,d)=(4,4,2)$ in \eqref{St-est}, we establish an interesting and unexpected connection with a combinatorial problem about counting words with four letters that dates back to the time of Strichartz's original work (see Appendix \ref{appendix}).

\section{Main Results}
The main abstract method underlying all the following results is to break the desired estimate into several simpler pieces, prove a sharp estimate for each piece and then obtain a sharp inequality for the full object. For the trained analyst, this strategy is most likely doomed to failure, but in our particular scenario the pieces are mutually orthogonal (and of finite dimension in most cases) and this allows us to avoid any loss of sharpness. This will become clear in the next sections.

\subsection{The Hermite Polynomial Approach}\label{hermite-case}
Let $\{H_m(x)\}_{m\geq 0}$ be the Hermite polynomials associated with the normal distribution $\d\ga(x)=(2\pi)^{-1/2}e^{-x^2/2}\dx$. In the sense of \cite[Chapters 2 and 5]{Sz}, these are the orthogonal polynomials associated with the measure $\d\ga(x)$ and normalized by the condition that each $H_m(x)$ is monic. For a given dimension $d$ and a given vector $\bm \in \Z^d_+$ ($\Z_+=\{0,1,2,...\}$) we write 
$$
H_{\bm} (\bx)=H_{m_1}(x_1){{...}} H_{m_d}(x_d),
$$ 
where $\bx=(x_1,{{...}},x_d)\in\R^d$. We also write 
$$
\d\ga_d(\bx)=(2\pi)^{-d/2}e^{-\|\bx\|^2/2}\d\bx,
$$
to denote the normal distribution in $\R^d$ where $\|\bx\|=\sqrt{x_1^2+{{...}}+x_d^2}$. 

It is known (see \cite[Chapter 5]{Sz}) that the multi-variate Hermite polynomials $\{H_{\bm} (\bx)\}$ form an orthogonal basis of $L^2(\R^d,\d\ga_d(\bx))$ and that
\est{
\int_{\R^d} |H_\bm(\bx)|^2\d\ga_d(\bx) = \bm!: = m_1!{{...}} m_d!.
}
As a consequence, it can be shown that the functions 
\es{\label{Phi-def}
\Phi_\bm(\bx) = H_\bm(\sqrt{4\pi}\,\bx)e^{-\pi\|\bx\|^2}
}
form an orthogonal basis of $L^2(\R^d,\d\bx)$. Thus, any function $f\in L^2(\R^d,\d\bx)$ can be uniquely written in the following form
\est{
f(\bx) = \sum_{\bm\in\Z^d_+} \alpha(\bm)\Phi_\bm(\bx).
}
We can now state our first main result. In what follows we define $|\bm|=m_1+m_2+{{...}}+m_d$, for any $\bm=(m_1,{...},m_d)\in \Z^d_+$.

\begin{theorem}\label{equiv-thm-hermite}
For any given $t\in\R$ define the following operator over $L^2(\R^d,\d\ga_d)$
\es{\label{herm-op-def}
\H^t:H_\bm(\bx)\mapsto e^{2\pi|\bm| it }H_\bm(\sqrt{2/p}\,\bx); \ \ \ \ \bx=(x_1,{{...}},x_d).
}
Let $2\leq p,q<\infty$ satisfy $\frac{d}{p}+\frac{2}{q}=\frac{d}{2}$ and let $f\in L^2(\R^d,\d\bx)$ have the following expansion $f(\bx)=\sum_{\bm\in\Z^d_+} \alpha(\bm)\Phi_\bm(\bx)$. Then we have
\est{
\frac{\|\|e^{it\Delta}f(\bx)\|_{L^p(\d\bx)}\|_{L^q(\d t)}}{(p^{-1/2p}2^{1/p-1/2})^d}=\left(\int_{-1/2}^{1/2}\bigg(\int_{\R^d}|\H^t T_{-i}g(\bx)|^p\d\ga_d(\bx)\bigg)^{q/p}\dt\right)^{1/q}
}
and
\est{
\frac{C(p,d)}{(p^{-1/2p}2^{1/p-1/2})^d}\|f(\bx)\|_{L^2(\d\bx)} = \bigg(\int_{\R^d}|g(\bx)|^2\d\ga_d(\bx)\bigg)^{1/2},
}
where $g(\bx)=\sum_{\bm\in\Z^d_+} \alpha(\bm)H_\bm(\bx)$, $C(p,d)$ is given in \eqref{sharp-const} and $T_{-i}$ is the operator defined in \eqref{T-om-op}.
\end{theorem}

\begin{corollary}
Gaussians maximize the Strichartz estimate \eqref{St-est} if and only if
\es{\label{equiv-est-hermite}
\left(\int_{-1/2}^{1/2}\bigg(\int_{\R^d}|\H^t g(\bx)|^p\d\ga_d(\bx)\bigg)^{q/p}\dt\right)^{1/q}\leq \bigg(\int_{\R^d}|g(\bx)|^2\d\ga_d(\bx)\bigg)^{1/2},
}
for all $g\in L^2(\R^d,\d\ga_d)$.
\end{corollary}

\noindent {\bf Remarks.}
\begin{enumerate}
\item Since the operator $T_{-i}$ is isometric and invertible on $L^2(\R^d,\d\ga_d)$, the corollary above easily follows from Theorem \ref{equiv-thm-hermite}. The factor $e^{2\pi|\bm| it }$ in the definition of the operator $\H^t$ introduces the possibility of using all the machinery from Fourier series to prove Conjecture \ref{St-conj} for all exponents. Unfortunately, the author was not able to achieve any  satisfactory result with the mentioned approach for non-even exponents. Nevertheless, we were able to exploit this approach in the known cases of even exponents and to give a new way to understand Conjecture \ref{St-conj}.

\item Theorem \ref{equiv-thm-hermite} was inspired by Beckner's approach for the sharp Hausdorff-Young inequality (see \cite{Be}). Roughly speaking, the heart of Beckner's proof relies on an application of the Central Limit Theorem to approximate the Hermite semi-group operator $H_n\mapsto \om^n H_n$ by normalized tensor products of a discrete version of the same operator in the two-point space. This strategy was generalized by the present author in \cite{Gon2}, where we showed that not only the Hermite semi-group can be approximated, but any operator given by a Gaussian kernel can be approximated by tensor products of any operator (not only Beckner's discrete operator in the two-point space) satisfying the right compatibility conditions. In the eyes of the author, the challenge presented by Conjecture \ref{St-conj}, within this framework, is to find a special way of discretizing the time variable.   
\end{enumerate}

As explained in the introduction, the known cases where the Strichartz inequality \eqref{St-est} is maximized by Gaussians are $(p,q,d)\in \{(6,6,1);(4,8,1);(4,4,2)\}$ and they all share the following property: $p=2k$ and $q=2k\ell$, for positive integers $k\geq 2$ and $\ell\geq 1$. In these situations, this relation allows us to use Theorem \ref{equiv-thm-hermite} and the orthogonality of Fourier series to transform the problem into an $(\ell^2\to \ell^2)$--estimate on a suitable space of sequences indexed by certain matrices. In what follows, we define the required mathematical objects and spaces for any given $k\geq 2$ and $\ell\geq 1$ and we explicitly calculate the resulting operators that naturally emerge when one tries to compute the left hand side of inequality \eqref{equiv-est-hermite}.

Let $g(\bx)=\sum_{\bm\in\Z^d_+} \alpha(\bm)H_\bm(\bx)$ be a function in $L^2(\R^d,\d\ga_d)$. Let $I=(-1/2,1/2)$ and $\la=\sqrt{2/p}=\sqrt{1/k}$. We have
\es{\label{gen-ident-hermite}
&\int_I\bigg(\int_{\R^d} |\H^tg(\bx)|^{2k}\d\ga_d(\bx)\bigg)^\ell\dt \\ 
& = \int_I\bigg(\int_{\R^d} \bigg| \sum_{\bm^1,...,\bm^k\in\Z_+^d} \prod_{j=1}^k \alpha(\bm^{j})   H_{\bm^j}(\la \bx) e^{2\pi (|\bm^1|+...+|\bm^k|) it }\bigg|^2 \d\ga_d(\bx)\bigg)^\ell \dt \\
& = \int_I\int_{(\R^d)^\ell} \bigg| \sum_{S\geq 0}\sum_{\stackrel{\sum_{i,j}|\bm^{i,j}|=S}{\bm^{i,j}\in\Z_+^d}} \prod_{i=1}^\ell \prod_{j=1}^k \alpha(\bm^{i,j})   H_{\bm^{i,j}}(\la \bx^i) e^{2\pi S it }\bigg|^2 \d\ga_d(\bx^1){...}\d\ga_d(\bx^\ell) \dt \\
& = \sum_{S\geq 0}\int_{(\R^d)^\ell} \bigg|  \sum_{\stackrel{\sum_{i,j}|\bm^{i,j}|=S}{\bm^{i,j}\in\Z_+^d}} \prod_{i=1}^\ell \prod_{j=1}^k \alpha(\bm^{i,j})   H_{\bm^{i,j}}(\la \bx^i)\bigg|^2 \d\ga_d(\bx^1){...}\d\ga_d(\bx^\ell) \\ 
& = \sum^*\prod_{i=1}^\ell \prod_{j=1}^k \alpha(\bm^{i,j})\ov{\alpha(\bn^{i,j})}\int_{(\R^d)^\ell} \prod_{i=1}^\ell \prod_{j=1}^k H_{\bm^{i,j}}(\la \bx^i)H_{\bn^{i,j}}(\la \bx^i)\d\ga_d(\bx^1){...}\d\ga_d(\bx^\ell),
}
where the summation $\sum^*$ is taken over all $\bm^{i,j},\bn^{i,j}\in\Z^d_+$ for $i=1,{{...}},\ell$ and $j=1,...,k$ such that $\sum_{i,j}|\bm^{i,j}| = \sum_{i,j}|\bn^{i,j}|$. The third identity above is due to the orthogonality of Fourier series $\{e^{2\pi iSt}\}_{S\in\Z}$. In an analogous way, we also have
\es{\label{gen-ident-hermite-2}
\bigg(\int_{\R^d}|g(\bx)|^2 \d\ga_d(\bx)\bigg)^{\ell k} = \sum_{\bm^{i,j}\in\Z^d_+} \prod_{i=1}^\ell \prod_{j=1}^k |\alpha(\bm^{i,j})|^2 M!,
}
where $M!=\prod_{i=1}^\ell \prod_{j=1}^k \bm^{i,j}!$ (recall that $(m_1,{{...}},m_d)!:= m_1!{{...}}m_d!$).

To be able to clearly analyze the resulting operator that appears in the last line of \eqref{gen-ident-hermite} we need to make some definitions first. Define $\M=\M^{\ell,k}_d$ as the space of $\ell\times k$ matrices $M=[\bm^{i,j}]$ for $i=1,{{...}},\ell$ and $j=1,{{...}},k$ where each entry is a vector $\bm^{i,j}\in\Z_+^d$. Next, define $\F=\F^{\ell,k}_d$ as the space of functions $\p:\M \to \C$ such that 
$$
\sum_{M\in\M} |\p(M)|^2M! < \infty.
$$
The space $\F$ is a Hilbert space of sequences indexed by the matrices in $\M$ and endowed with the following inner product
\es{\label{inner-prod-LL-hermite}
\langle \p,\psi \rangle_\F = \sum_{M\in\M} \p(M)\ov{\psi(M)}M!.
}
We define an operator $\PP=\PP^{\ell,k}_d$ that maps a function $\p:\M\to\C$ into a function $\psi:\M\to\C$ 
by
\es{\label{operator-P-hermite}
\psi(M)=\PP\p(M) = \sum_{\stackrel{N\in\M}{|N|=|M|}} \p(N
) \frac{P(M,N)}{M!},
}
where $M=[\bm^{i,j}]\in \M$ and $|M|=\sum_{i,j} |\bm^{i,j}|$. The coefficients above are given by
\es{\label{P-coeff}
P(M,N) & =  \int_{(\R^d)^\ell} \bigg\{\prod_{i=1}^\ell\prod_{j=1}^kH_{\bm^{i,j}}(\la \bx^i)H_{\bn^{i,j}}(\la \bx^i)\bigg\}\d\ga_d(\bx^1){{...}}\d\ga_d(\bx^\ell)
}
for all $M=[\bm^{i,j}]$ and $N=[\bn^{i,j}]$. For each $S\geq 0$, let $\F_S$ denote the closed subspace of functions $\p:\M\to\C$ such that $\p(M)=0$ for any matrix $M\in\M$ such that $|M|\neq S$. Note that $\dim (\F_S)<\infty$ and $\PP(\F_S)\subset \F_S$. Also, the spaces $\F_S$ are orthogonal to each other with respect to the inner product \eqref{inner-prod-LL-hermite} and 
$$
\F=\ov{\bigoplus_{S\geq 0} \F_S}.
$$ 

Our next main result concerns the operator $\PP$, its norm over the space $\F$ and the its relation with estimate \eqref{equiv-est-hermite}. Since the coefficients $P(M,N)$ are real, clearly $\PP$ is always symmetric. Depending on the exponents $\ell,k$ and $d$, the operator $\PP$ may be unbounded (hence, not defined in the whole $\F$). It is the goal of our next result to give a full description of these scenarios.

\begin{theorem}\label{L2-L2-estimate-hermite}
Let $\ell\geq 1$, $k\geq 2$ and $d\geq 1$ be integers and consider the operator $\PP$ defined in \eqref{operator-P-hermite} acting on the space $\F$. Then
\begin{equation}\label{Integral-id-hermite}
\left\{
\begin{array}{lc}
\displaystyle\int_I\bigg(\int_{\R^d} |\H^tg(\bx)|^{2k}\d\ga_d(\bx)\bigg)^\ell\dt  = \langle \p,\PP\p\rangle_{\F} \\ 
\bigg(\int_{\R^d} |g(\bx)|^2 \d\ga_d(\bx)\bigg)^{k\ell}  = \|\p\|^2_\F,
\end{array}
\right.
\end{equation}
if $g(\bx)=\sum_{\bm} \alpha(\bm)H_\bm(\bx)$  and if $\p:\M\to\C$ is defined by $\p(M)=\prod_{i,j} \alpha(\bm^{i,j})$ for any $M=[\bm^{i,j}]\in\M$. Let $S\geq 0$ be an integer and let $\PP_S$ be the restriction to $\F_S$ of the operator $\PP$ $($recall that $\PP(\F_S)\subset \F_S$\nonumber$)$. Then $\PP_S^2=\PP_S$ $($hence $\PP_S$ is a projection$)$ if and only if $(k-1)\ell d/2=1$. In this case, $\PP$ is well-defined in the whole $\F$ and it is also a projection. In general, if $\mu:=(k-1)\ell d/2>1$ we have
\est{
\|\PP_S\|_{\F_S\to\F_S}=\frac{\mu (\mu+1){{...}}(\mu+\lfloor S/2\rfloor-1)}{\lfloor S/2\rfloor!} \sim \frac{S^{\mu-1}}{2^{\mu-1}\Gamma(\mu)}, \ \ S\to\infty.
}
In particular, $\PP$ is not bounded in $\F$ for $\mu>1$.
\end{theorem}

\noindent {\bf Remark.} 
For $k\geq 2$, the only case where $0< \mu<1$ is $(k,\ell,d)=(2,1,1)$. This is a pathological case that appears due to the presence of probability measures and the possibility of using Jensen's inequality, and that is why we exclude it from our analysis. Considering the situation $p=2k$ and $q=2k\ell$, where $k\geq 2$ and $\ell\geq 1$ are integers, the three cases where Conjecture \ref{St-conj} is known to be true $(p,q,d)\in\{(6,6,1);(4,8,1);(4,4,2)\}$ match exactly with those where $\mu=(k-1)\ell d/2=1$.

\begin{corollary}\label{cor}
Gaussians maximize the Strichartz inequality \eqref{St-est} for $(p,q,d)\in\{(6,6,1);(4,8,1);(4,4,2)\}$.
\end{corollary}

\subsection{The Laguerre Polynomial Approach}
For any $\nu>-1$ we denote by $\{L_n^{(\nu)}(x)\}_{n\geq 0}$ the generalized Laguerre polynomials associated with the parameter $\nu$. In the sense of \cite[Chapters 2 and 5]{Sz}, these are the orthogonal polynomials associated with the measure $e^{-x}x^\nu \dx$ ($x>0$) and normalized by the condition
$$
\int_0^\infty |L_n^{(\nu)}(x)|^2 e^{-x}x^\nu \dx = \frac{\Gamma(n+\nu+1)}{n!}.
$$
They are also known to form an orthogonal basis in the space $L^2(\R_+,e^{-x}x^\nu \dx)$. In analogy with the Hermite polynomials, it can be shown that for a given dimension $d$ the functions
\es{\label{Psi-def}
\Psi_n(\bx) = L_n^{(\nu)}(2\pi \|\bx\|^2) e^{-\pi\|\bx\|^2},
} 
with $\nu=d/2-1$ form an orthogonal basis of the space of radial functions in $L^2(\R^d,\d\bx)$. Thus, any radial function $f\in L^2(\R^d,\d\bx)$ can be  uniquely written in the form
$$
f(\bx) = \sum_{n\geq 0} \alpha(n)\Psi_n(\bx).
$$

Our next main result shows an analogue of Theorem \ref{equiv-thm-hermite} but only for radial initial data.

\begin{theorem}\label{equiv-thm-laguerre}
Let $d$ be a given dimension and set $\nu=d/2-1$. For any given $t\in\R$ define the following operator on $L^2\left(\R_+, \frac{x^{\nu}e^{-x}}{\Gamma(\nu+1)}\dx\right)$
\est{
\LL^t:L_n^{(\nu)}(x)\mapsto e^{2\pi n it }L_n^{(\nu)}\big(\tfrac{2}{p}\,x\big).
}
Let $2\leq p,q<\infty $ satisfy $\frac{d}{p}+\frac{2}{q}=\frac{d}{2}$ and let $f\in L^2(\R^d,\d\bx)$ be a radial function having the following expansion $f(\bx)=\sum_{n\geq 0} \alpha(n)\Psi_n(\bx)$. Then we have
\es{\label{St-est-semigroup-laguerre}
\frac{\|\|e^{it\Delta}f(\bx)\|_{L^p(\d\bx)}\|_{L^q(\d t)}}{(p^{-1/2p}2^{1/p-1/2})^d}=\left(\int_{-1/2}^{1/2}\bigg(\int_0^\infty|\LL^tg(x)|^p \frac{e^{-x}x^\nu}{\Gamma(\nu+1)}\bigg)^{q/p}\dt\right)^{1/q}
}
and
\est{
\frac{C(p,d)}{(p^{-1/2p}2^{1/p-1/2})^d}\|f(\bx)\|_{L^2(\d\bx)} = \bigg(\int_0^\infty|g(x)|^2\frac{e^{-x}x^\nu}{\Gamma(\nu+1)}\d x\bigg)^{1/2},
}
where $g(\bx)=\sum_{n\geq 0} \alpha(n)L^{(\nu)}_n(x)$ and $C(p,d)$ is given in \eqref{sharp-const}.
\end{theorem}

\noindent {\bf Remark.} 
The above result allows us to explore again the cases where the exponents are even, giving us the opportunity of replicating the results of the Hermite case. However, the resulting operators are not projections any more and the methods used to analyze the Hermite case (see the proof of Theorem \ref{L2-L2-estimate-hermite}) do not work. Roughly speaking, the main reason why they do not work is that the Poisson Kernel associated with Laguerre polynomials is given by the Bessel functions $I_\nu(z)$, which are wild creatures that do not enjoy being handled.

\subsubsection{{\bf The Case $(p,q,d)=(4,4,2)$}} 
We start by calculating the resulting operator that arises from the right hand side term of \eqref{St-est-semigroup-laguerre}. For $d=2$ we have $\nu=d/2-1=0$ and we write $L_n(x)=L_n^{(0)}(x)$ to simplify notation. If $g(x)=\sum_{n\geq 0} \alpha(n)L_n(x)$ we deduce that
\es{\label{crucial-id-laguerre}
& \int_{-1/2}^{1/2}\int_0^\infty|\LL^tg(x)|^4 e^{-x}\d x\dt \\ &
= \sum_{a,b\geq 0} \alpha(a)\alpha(b)\bigg\{\sum_{\stackrel{c,d\geq 0}{c+d=a+b}}\ov{\alpha(c)}\,\ov{\alpha(d)}\int_0^\infty L_a(x/2)L_b(x/2)L_c(x/2)L_d(x/2)e^{-x}\d x \bigg\}
}
and
\es{\label{crucial-id-laguerre-2}
\bigg(\int_0^\infty|g(x)|^2 e^{-x}\d x\bigg)^2 = \sum_{n,m\geq 0} |\alpha(n)\alpha(m)|^2.
}
The above calculations suggest the following definitions. Let $\G=\ell^2(\Z_+^2)$ be the standard Hilbert space of complex-valued sequences $\{\p(a,b)\}_{(a,b)\in\Z_+^2}$ such that
$$
\|\p\|_\G^2:=\sum_{a,b\geq 0} |\p(a,b)|^2 < \infty.
$$
The inner product in $\G$ is given by
$$
\langle \p,\psi \rangle_\G = \sum_{a,b\geq 0} \p(a,b)\ov{\psi(a,b)}.
$$
Next, we define an operator $\Q$ for any given $\p\in \G$ by 
\est{
\Q\p(a,b) = \sum_{\stackrel{c,d\geq 0}{c+d=a+b}} \p(c,d)Q(a,b,c,d),
}
where
\es{\label{Q-coeff}
Q(a,b,c,d)=\int_0^\infty L_a(x/2)L_b(x/2)L_c(x/2)L_d(x/2)e^{-x}\d x.
}
In analogy with the Hermite polynomial approach, we can analyze the operator $\Q$ by its action in certain orthogonal invariant subspaces of finite dimension. For any integer $S\geq 0$, let $\G_S$ denote the subspace of sequences $\p:\Z^2_+\to \C$ such that $\p(a,b)=0$ if $a+b\neq S$. Clearly, the collection $\{\G_S\}_{S\geq 0}$ is orthogonal and 
$$
\G = \ov{\bigoplus_{S\geq 0} \G_S}.
$$
We also have that $\dim(\G_S)=S+1$ and $\Q(\G_S)\subset \G_S$. Letting $\Q_S$ denote the restriction of $\Q$ to the subspace $\G_S$, we conclude that the operator $\Q_S$ can be represented by the following matrix 
\es{\label{matrix-rep-Q}
\Q_S=[Q(a,S-a,c,S-c)]_{a,c=0,{{...}},S}.
}

\begin{theorem}\label{QS-op-thm}
For any radial $f(\bx)=\sum_{n\geq 0} \alpha(n)\Psi_n(x) \in L^2(\R^2,\d\bx)$ we have
\begin{equation}\label{int-id-G}
\|e^{it\Delta}f(\bx)\|^4_{L^4(\R^3,\d\bx\dt)} = \frac{1}{16}\langle \p, \Q\p \rangle_\G
\end{equation}
and
\begin{equation}\label{int-id-G-2}
\big(2^{-1/2}\|f(\bx)\|_{L^2(\R^2,\d\bx)}\big)^4 = \frac{1}{16}\|\p\|_\G^2,
\end{equation}
where $\p(a,b)=\alpha(a)\alpha(b)$ for all $a,b\geq 0$. For any $S\geq 0$ the matrix $\Q_S$ at \eqref{matrix-rep-Q} is a positive semi-definite doubly stochastic matrix with strictly positive entries. We conclude that $\|\Q\|_{\G\to\G}=1$. Furthermore, a function $\p\in \G$ satisfies 
$$
\langle \p, \Q\p \rangle_\G = \|\p\|_\G^2
$$
if and only if it has the property that $\p(a,b)=\p(c,d)$ whenever $a+b=c+d$.
\end{theorem}

\begin{corollary}\label{radial-St-sharp}
For any radial $f(\bx)\in L^2(\R^2,\d\bx)$ we have
\begin{equation}\label{radial-St-ineq}
\|e^{it\Delta}f(\bx)\|_{L^4(\R^3,\d\bx\dt)} \leq 2^{-1/2}\|f(\bx)\|_{L^2(\R^2,\d\bx)},
\end{equation}
and equality is attained if and only if $f(\bx)=Ae^{-B\|\bx\|^2}$, where $A,B\in\C$ and $\re B>0$.
\end{corollary}

\subsection{Spherical Harmonics and Gegenbauer Polynomials}
In this part we make use of the special connection between spherical harmonics and Gegenbauer polynomials given by the Funk-Hecke formula \eqref{funk-hecke-formula} to prove a sharpened inequality for the Schr\"odinger operator that is maximized by radial functions. We perform a diagonalization process in an operator over $L^2(\Sp^{d-1})$ that naturally emerges from our approach and which ultimately allows us to perform a near-extremizer analysis.

For any $d\geq 3$ define the following operator
\es{\label{radial-operator}
R(g)(\bxi) = \int_{\Sp^{d-1}} {g(\bze)} \frac{\d\bze}{\|\bxi-\bze\|^{d-2}}
}
for $g\in L^2(\Sp^{d-1})$. Above, $\Sp^{d-1}$ represents the unit sphere in $\R^d$, $\|\cdot\|$ the Euclidean norm in $\R^d$ and $\d\bze$ (and $\d\bxi$ below) the natural surface measure over $\Sp^{d-1}$.

\begin{theorem}\label{radial-id-thm}
Let $d\geq 3$. Let $f\in L^2(\R^d,\d\bx)$ be a function of Schwartz class and define $g(r,\bxi) = \ft f(r\bxi)$ for any $r>0$ and $\bxi\in \Sp^{d-1}$. We obtain
\es{\label{radial-identi-operat-R}
\int_\R \int_{\R^d} |e^{it\Delta}f(\bx)|^2\frac{\d\bx}{\|\bx\|^2}\,\dt = \frac{\pi}{(d-2)|\Sp^{d-1}|}\int_0^\infty \int_{\Sp^{d-1}} g(r,\bxi){\ov{R(g(r,\cdot))(\bxi)}}\d\bxi r^{d-1}\d r.
}
Moreover, the operator $R$ is bounded over $L^2(\Sp^{d-1})$ and for all $g\in L^2(\Sp^{d-1})$ we have
\es{\label{R-sharp-ineq}
\int_{\Sp^{d-1}} g(\bxi) \ov{R(g)(\bxi)} \d\bxi \leq |\Sp^{d-1}|\bigg\{\int_{\Sp^{d-1}}  |g(\bxi)|^2 \d\bxi - \frac{2}{d}{\rm Dist}(g,{\rm Const})^2\bigg\},
}
where ${\rm Dist}(g,{\rm Const})$ denotes the distance in the $L^2(\Sp^{d-1})$-norm of $g(\bxi)$ to the subspace of constant functions.
\end{theorem}

\begin{theorem}\label{radial-thm}
Let $d\geq 3$. Then for all $f\in L^2(\R^d,\d\bx)$ we have
\es{\label{radial-ineq}
\int_\R \int_{\R^d} |e^{it\Delta}f(\bx)|^2 \frac{\d\bx}{\|\bx\|^2}\dt \leq \frac{\pi}{d-2}\bigg\{\int_{\R^d} |f(\bx)|^2 \d\bx - \frac{2}{d}{\rm Dist}(f,\,{\rm Radial})^2\bigg\},
}
where ${\rm Dist}(f,\,{\rm Radial})$ denotes the distance in the $L^2(\R^d,\d\bx)$-norm of $f(\bx)$ to the subspace of radial functions. In particular, we have
\es{\label{radial-ineq-2}
\int_\R \int_{\R^d} |e^{it\Delta}f(\bx)|^2 \frac{\d\bx}{\|\bx\|^2}\dt \leq \frac{\pi}{d-2}\int_{\R^d} |f(\bx)|^2 \d\bx,
}
and equality is attained if and only if $f(\bx)$ is a radial function.
\end{theorem}

\begin{corollary}\label{radial-cor}
Let $d\geq 3$ and $2\leq p \leq 2+4/d$. There exists $C>0$ such that for all $f\in L^2(\R^d,\d\bx)$ we have
\est{
\bigg\{\int_\R \int_{\R^d} |e^{it\Delta}f(\bx)|^p \frac{\d\bx}{\|\bx\|^{d+2-pd/2}}\dt \bigg\}^{1/p}\leq C\bigg\{\int_{\R^d} |f(\bx)|^2 \d\bx\bigg\}^{1/2}.
}
\end{corollary}

\noindent {\bf Remark.} 
\begin{enumerate}
\item The above corollary is a straightforward consequence of Stein and Weiss interpolation result with change of measures \cite[Theorem 2.11]{SW} $($which works very well for homogeneous weights$)$ in conjunction with Theorem \ref{radial-thm} and the Strichartz estimate \eqref{St-est} for $p=q=2+4/d$. To the best of our knowledge the first time that inequality \eqref{radial-ineq-2} appeared was in \cite{Sim}. In \cite{Wat}, Watanabe identified the extremizers as radial functions (and also for related inequalities for the fractional laplacian). This fact was later rediscovered (in a much general framework) by Bez and Sugimoto in \cite{BS}. Our improvement here lies in the near extremizer analysis of \eqref{radial-ineq}.

\item The fact that \eqref{radial-ineq-2} is attained for any radial function is a direct consequence of \eqref{radial-identi-operat-R} and the fact that $R(\bo 1)\equiv |S^{d-1}|\bo 1$ (this can be shown using \eqref{funk-hecke-formula} and \eqref{magic-formula-Geg} for $n=0$ and $a=\nu$). We also note that the proofs of Theorems \ref{radial-id-thm} and \ref{radial-thm} actually show that the difference between the left and right hand side in \eqref{radial-ineq-2} is proportional to ${\rm Dist}(f,\,{\rm Radial})^2$.
\end{enumerate}

\section{Proofs for the Main Results}

Throughout this paper we use the following definition for the Fourier Transform of a function $f(\bx)$
\es{\label{FT}
\ft f(\by)  = \int_{\R^d} f(\bx)e^{-2\pi i\bx \cdot \by}\d\bx.
}

\subsection{The Hermite and Laguerre Polynomials Part}

To prove Theorems \ref{equiv-thm-hermite} and \ref{equiv-thm-laguerre} we start by calculating the solution of the Schr\"odinger equation for the functions $\Phi_{\bm}(\bx)$ and $\Psi_{n}(\bx)$ defined in \eqref{Phi-def} and \eqref{Psi-def} respectively.

\begin{lemma}\label{SE-flow-lague-herm}
For all $\bm\in\Z^d_+$ we have
\es{\label{hermite-flow-id}
& e^{i\Delta t} (\Phi_{\bm})(\bx) \\ & = (1+4\pi it)^{-d/2}\left(\sqrt{\frac{1-4\pi it}{1+4\pi it}}\right)^{|\bm|} \Phi_\bm\left(\frac{\bx}{\sqrt{1+16\pi^2 t^2}}\right)\exp\left[\frac{4\pi^2it}{1+16\pi^2 t^2}\|\bx\|^2\right].
}
Also, for all $n\geq 0$ we have
\es{\label{laguerre-flow-id}
& e^{i\Delta t} (\Psi_{n})(\bx) \\ & = (1+4\pi it)^{-d/2}\left(\frac{1-4\pi it }{1+4\pi it}\right)^{n} \Psi_n\left(\frac{\bx}{\sqrt{1+16\pi^2 t^2}}\right)\exp\left[\frac{4\pi^2 it}{1+16\pi^2 t^2}\|\bx\|^2\right].
}
\end{lemma}

\begin{proof}
We prove first the second identity. Firstly, if $f(\bx)=f(\|\bx\|)$ is a radial function, then its Fourier transform $\widehat{f}(\by)=\ft f(\|\by\|)$ is also radial and we have
\begin{equation}\label{radial-fourier}
s^{\nu}\widehat{f}(s) =2\pi \int_0^\infty  f(r)J_\nu(2\pi rs)r^{\nu+1}\d r, 
\end{equation}
for every $s>0$, where $\nu=d/2-1$ and $J_\nu(z)$ is the Bessel function of the first kind. Secondly, the identity in \cite[7.421-4]{GR} states that
\begin{equation*}
\int_0^\infty x^{\nu+1}e^{-\beta x^2}L^{(\nu)}_n\big(\alpha x^2\big)J_\nu(xy)dx=(2\beta)^{-\nu-1}\bigg(1-\frac{\alpha}{\beta}\bigg)^n y^\nu e^{-\frac{y^2}{4\beta}}L_n^{(\nu)}\bigg(\frac{\alpha y^2}{4\beta(\alpha-\beta)}\bigg)
\end{equation*}
for any $\alpha,\beta\in\C$ with $\re \beta>0$ and any $\nu>-1$. Applying the above identity for $\alpha=2\pi$ and $\beta=\pi(a+1)$ in conjunction with identity \eqref{radial-fourier} we deduce that
\es{\label{Laguerre-a-id}
\ft{\bigg(L^{(\nu)}_n(2\pi\|\bx\|^2)e^{-\pi (a+1)\|\bx\|^2}\bigg)}(\by) = (1+a)^{-d/2}\bigg(\frac{a-1}{a+1}\bigg)^{n} L_n^{(\nu)}\bigg(\frac{2\pi\|\by\|^2}{1-a^2}\bigg)e^{-\frac{\pi\|\by\|^2}{1+a}},
}
where $\nu=d/2-1$. Taking $a=0$ in the above identity reads $\ft \Psi_n (\by)= (-1)^n\Psi_n(\by)$. Finally, since
\est{
\ft{(e^{i\Delta t} \Psi_{n})}(\by) = (-1)^n e^{4\pi^2 it\|\by\|^2} \Psi_{n}(\by) = (-1)^n L^{(\nu)}_n(2\pi\|\by\|^2)e^{-\pi (4\pi it+1)\|\by\|^2},
}
we can use identity \eqref{Laguerre-a-id} with $a=4\pi it$ to deduce \eqref{laguerre-flow-id}.

Identity \eqref{hermite-flow-id} in dimension $d>1$ follows from its one-dimensional version since we have $\Phi_\bm(\bx)=\Phi_{m_1}(x_1){{...}}\Phi_{m_d}(x_d)$ if $\bm=(m_1,{{...}},m_d)$. We can now use identities \cite[7.388--2 and 7.388--4]{GR} to show that
\es{\label{hermite-a-id}
\ft{\bigg(H_n(\sqrt{4\pi}x)e^{-\pi (a+1)x^2}\bigg)}(y) = (1+a)^{-d/2}\bigg(-\sqrt{\frac{a-1}{a+1}}\bigg)^{n} H_n\bigg(\frac{\sqrt{4\pi}y}{\sqrt{1-a^2}}\bigg)e^{-\frac{\pi y^2}{1+a}},
}
for any $a\in\C$ with $\re a >-1$. Finally, identity \eqref{hermite-flow-id} for $d=1$ follows by an analogous way as in the Laguerre polynomial case, but now using identity \eqref{hermite-a-id} for $a=4\pi it$ and the fact that $\ft{(e^{i\Delta t} \Phi_{n})}(\by) = (-i)^n e^{-4\pi^2 it\|\by\|^2} \Phi_{n}(\by)$ (which can also be deduced from \eqref{hermite-a-id} by taking $a=0$).
\end{proof}

\begin{proof}[\bf Proof of Theorem \ref{equiv-thm-hermite}]
\noindent {\it Step 1.}
Let $f\in L^2(\R^d,\d\bx)$ have the following expansion
$$
f(\bx) = \sum_{\bm\in\Z^d_+}\alpha(\bm) \Phi_\bm(\bx)
$$
and define $g(\bx)=\sum_{\bm\in\Z^d_+}\alpha(\bm) H_\bm(\bx)$. By Lemma \ref{SE-flow-lague-herm}, we obtain that
\est{
|e^{i\Delta t}f(\bx)|= |1+16\pi^2 t^2|^{-d/4}\bigg|\sum_{\bm\in\Z^d_+} \alpha(\bm) \left(\sqrt{\frac{1-4\pi it }{1+4\pi it}}\right)^{|\bm|} \Phi_\bm\left(\frac{\bx}{\sqrt{1+16\pi^2 t^2}}\right)\bigg|.
}
Recalling that $\Phi_\bm(\bx) = H_\bm(\sqrt{4\pi}\,\bx)e^{-\pi\|\bx\|^2}$ and $\la=\sqrt{2/p}$, we deduce that
\es{\label{hermite-change-variables}
& \int_\R \left(\int_\R |e^{i\Delta t}f(\bx)|^p\d\bx\right)^{q/p}\dt 
\\ & = \int_\R \left(\int_{\R^d}\left|\sum_\bm \alpha(\bm)\left(\sqrt{\frac{1-4\pi it}{1+4\pi it}}\right)^{|\bm|}\hspace{-2mm}\Phi_\bm\left(\frac{\bx}{\sqrt{1+16\pi^2 t^2}}\right)\right|^p\d\bx\right)^{q/p}\hspace{-4mm}\frac{\dt}{(1+16\pi^2 t^2)^{\frac{qd}{4}}}
\\ & =p^{-qd/2p}\int_\R \left(\int_{\R^d}\left|\sum_\bm \alpha(\bm)\left(\sqrt{\frac{1-4\pi it}{1+4\pi it}}\right)^{|\bm|} H_\bm(\la\by)\right|^p\d\ga_d(\by)\right)^{q/p}\frac{\dt}{1+16\pi^2 t^2} \\
& =(p^{-qd/2p}/4)\int_{-1/2}^{1/2} \left(\int_{\R^d}\left|\sum_\bm \alpha(\bm)e^{\pi i|\bm|s} H_\bm(\la\by)\right|^p\d\ga_d(\by)\right)^{q/p}\ds \\
& = (p^{-qd/2p}/4)\int_{-1/2}^{1/2}\|\H^{t/2} g\|^q_{L^p(\d\ga_d)}\dt,
}
where in the second identity above we did the change of variables $\by=\sqrt{\frac{2\pi p}{1+16\pi^2t^2}}\,\bx$ and used that $\frac{qd}{4}-\frac{qd}{2p}=1$, and in the third identity we did $\pi s=\arctan(-4\pi t)$. Similarly, we deduce that
\es{\label{hermite-change-variables-2}
\int_{\R^d} |f(\bx)|^2\d\bx = 2^{-d/2}\int_{\R^d} |g(\by)|^2\d\ga_d(\by).
}

\noindent {\it Step 2.} We now need to define an auxiliary linear operator $T_\om: L^2(\R^d,\d\ga_d)\to L^2(\R^d,\d\ga_d)$ for any $\om\in\C$ with $|\om|\leq 1$ as follows
\es{\label{T-om-op}
T_\om(H_\bm)(\bx) = \om^{|\bm|}H_\bm(\bx).
}
This defines a group with respect to complex multiplication: $T_{\om_1\om_2}=T_{\om_1}T_{\om_2}$ and $T_\om$ is an isometric transformation if $|\om|=1$. Now, observe that by the definition of $\H^t$ in \eqref{herm-op-def} we have $\H^{t+s} = \H^t T_{e^{2\pi i s}}$ for all real $s,t$ and, since $H_n(-x)=(-1)^nH_n(x)$, we also have $T_{-1}g(\bx) = g(-\bx)$ for all $g(x)$ and we conclude that $\H^tT_{-1} = T_{-1}\H^t$ for all real $t$. All these considerations imply that $\H^{t-1/4}T_{-i} = T_{-1} \H^t$ and $\H^{t+1/4}T_{-i}=\H^t$ for all real $t$. For $g(\by)=\sum_\bm \alpha(\bm) H_\bm(\by)$ we obtain 
\es{\label{herm-T-trick}
& \int_{-1/2}^{1/2}\|\H^{t}T_{-i}g\|^q_{L^p(\d\ga_d)}\dt \\ & = \frac{1}{2}\int_{-1/2}^{1/2}\|\H^{t/2+1/4}T_{-i}g\|^q_{L^p(\d\ga_d)}\dt + \frac{1}{2}\int_{-1/2}^{1/2}\|\H^{t/2-1/4}T_{-i}g\|^q_{L^p(\d\ga_d)}\dt \\
& = \frac{1}{2}\int_{-1/2}^{1/2}\|\H^{t/2} g\|^q_{L^p(\d\ga_d)}\dt + \frac{1}{2}\int_{-1/2}^{1/2}\|T_{-1}\H^{t/2} g\|^q_{L^p(\d\ga_d)}\dt \\& = \int_{-1/2}^{1/2}\|\H^{t/2} g\|^q_{L^p(\d\ga_d)}\dt.
}
The Theorem \ref{equiv-thm-hermite} follows from \eqref{hermite-change-variables}, \eqref{hermite-change-variables-2} and \eqref{herm-T-trick}. 
\end{proof}

Before we prove Theorem \ref{L2-L2-estimate-hermite} we need a basic lemma about self-adjoint linear transformations and their norms. The proof can be done using the Spectral Theorem and we leave the details to the interested reader.

\begin{lemma}\label{norm-lemma}
Let $\OO: \B \to \B$ be a bounded self-adjoint linear transformation and $\B$ be a separable Hilbert space over $\C$ with Hermitian inner product $\langle \cdot,\cdot\rangle_\B$. Assume that there exists a real number $\theta>0$ such that
$$
|\langle \OO^n \bu,\bv\rangle_\B| \ll \theta^n
$$
for all $n\geq 1$ and all $\bu,\bv\in\B$, where the implied constant depends only on $\bu$ and $\bv$. Then
$$
\|\OO\|_{\B\to\B} := \sup_{\bu\neq 0} \frac{\|\OO\bu\|_\B}{\|\bu\|_\B}\leq \theta.
$$
Moreover, if in addition we have
$$
|\langle \OO^n \bu_0,\bv_0\rangle_\B|\gg\theta^n
$$
for some vectors $\bu_0$ and $\bv_0$ and for all $n\geq 1$ then
$$
\|\OO\|_{\B\to\B} = \theta.
$$
\end{lemma}

The linear operators $T_\om$ defined in \eqref{T-om-op} for $|w|\leq 1$  also play an important role in the proof of Theorem \ref{L2-L2-estimate-hermite}. These operators are given by the following Mehler Kernel (see \cite[p. 163]{Be})
\es{\label{Mehler-kernel}
T_\om(\bx,\by) & = \frac{1}{(1-\om^2)^{d/2}} \exp\bigg[\frac{-\om^2(\|\bx\|^2 + \|\by\|^2) }{2(1-\om^2)} + \frac{\om \bx\cdot\by}{1-\om^2}\bigg] \\ & =
\sum_{\bm\in \Z^d_+}\om^{|\bm|} \frac{H_\bm(\bx)H_\bm(\by)}{\bm!},
}
where the convergence of the above series (for fixed $\bx$ and $\by$) is absolute for $|\om|<1$. In other words, for all $g\in L^2(\R^d,\d\ga_d)$ we have
$$
T_\om g(\bx) = \int_{\R^d} T_\om(\bx,\by)g(\by)\d\ga_d(\by).
$$
At this point we recommend the reader to recall the notation introduced in Section \ref{hermite-case}.

\begin{proof}[\bf Proof of Theorem \ref{L2-L2-estimate-hermite}]

\noindent {\it Step 1.} The identities in \eqref{Integral-id-hermite} easily follow from the definitions of the space $\F$ and the operator $\PP$ in conjunction with identities \eqref{gen-ident-hermite} and \eqref{gen-ident-hermite-2}. 

\noindent {\it Step 2.}  Let $S\ge 0$. Our goal is to explicitly compute the $n$-fold composition $\PP^n$ of the operator $\PP$ for any $n\geq 0$. In general, for any $n\geq 0$, if $\p\in\F$ and $M\in\M$ with $|M|=S$ we have
\es{\label{Pn-op}
\PP^n\p(M) = \sum_{\stackrel{N\in\M}{|N|=S}}\p(N) \frac{P^{(n)}(M,N)}{M!},
}
for some coefficients $P^{(n)}(M,N)$. The idea of the proof is the following: (a) Give a nice representation of these coefficients $P^{(n)}(M,N)$ in terms of a certain multiplication operator; (b) Show that $P^{(2)}(M,N)=P(M,N)$ in the case $(k-1)d\ell/2 = 1$, hence $\PP^2=\PP$; (c) Use Lemma \ref{norm-lemma} to exactly compute the norm of $\PP$ restricted to $\F_S$ when $\mu = (k-1)d\ell/2\neq 1$.

We start with $n=2$. In this case, if $\p\in\F$ and $M\in\M$ with $|M|=S$ we have
\es{\label{P2-op}
\PP^2\p(M) & = \sum_{\stackrel{N\in\M}{|N|=S}} \PP\p(N)\frac{P(M,N)}{M!} = \sum_{\stackrel{N\in\M}{|N|=S}} \bigg(\sum_{\stackrel{L\in\M}{|L|=S}}\p(L)\frac{P(N,L)}{N!}\bigg)\frac{P(M,N)}{M!} \\ 
& =  \sum_{\stackrel{L\in\M}{|L|=S}}\p(L) \frac{P^{(2)}(M,L)}{M!},
}
where
\es{\label{P2-coeff}
P^{(2)}(M,L) = \sum_{\stackrel{N\in\M}{|N|=S}} \frac{P(M,N)P(N,L)}{N!}.
}
To investigate the above coefficients we need to define a new kernel involving Hermite polynomials. For any $M=[\bm^{i,j}]\in \M$ and any collection of vectors $\{\bx^1,{{...}},\bx^\ell\}$ in $\R^d$ let
\est{
H_M(\bx^1,{{...}},\bx^\ell) = \prod_{i=1}^\ell\prod_{j=1}^k H_{\bm^{i,j}}(\la\bx^i)
}
(recall that $\la=\sqrt{2/p}=\sqrt{1/k}$). For any $S\geq 0$ define the following kernel
\es{\label{KS-kernel}
K_S(\bx^1,{{...}},\bx^\ell;\by^1,{{...}},\by^\ell)  :=\sum_{\stackrel{M=[\bn^{i,j}]\in \M}{|M|=S}} \frac{H_M(\bx^1,{{...}},\bx^\ell)H_M(\by^1,{{...}},\by^\ell)}{M!}.
}
Also define the associated operator
\est{
\K_S[g](\bx^1,{{...}},\bx^\ell) = \int_{(\R^d)^\ell} K_S(\bx^1,{{...}},\bx^\ell;\by^1,{{...}},\by^\ell) g(\by^1,{{...}},\by^\ell)\d\ga_d(\by^1){{...}}\d\ga_d(\by^\ell),
}
for any $g\in L^2((\R^d)^\ell,\d\ga_d(\by^1){{...}}\d\ga_d(\by^\ell))$. Using \eqref{P-coeff}, \eqref{P2-op} and \eqref{P2-coeff}, we obtain that
\est{
P^{(2)}(M,L)  = \int_{(\R^d)^\ell}  H_M(\bx^1,{{...}},\bx^\ell)\K_S[H_L](\bx^1,{{...}},\bx^\ell)\d\ga_d(\bx^1){{...}}\d\ga_d(\bx^\ell)
}
for any $M=[\bm^{i,j}]$ and $L=[\bl^{i,j}]$ in $\M$. In an analogous way, if $\p\in\F$ and $M\in\M$ with $|M|=S$ we have the following representation for the coefficient in \eqref{Pn-op}
\es{\label{Pn-coeff}
P^{(n)}(M,L) = \int_{(\R^d)^\ell}  H_M(\bx^1,{{...}},\bx^\ell)\K_S^{(n-1)}[H_L](\bx^1,{{...}},\bx^\ell)\d\ga_d(\bx^1){{...}}\d\ga_d(\bx^\ell),
}
where $\K_S^{(n-1)}$ is the $(n-1)$-fold composition of $\K_S$.

\noindent {\it Step 3.} Let $\mu=(k-1)d\ell/2$. We claim that the kernel $K_S$ in \eqref{KS-kernel} has the following alternative form
\es{\label{KS-alt-form}
& K_S(\bx^1,{{...}},\bx^\ell;\by^1,{{...}},\by^\ell) \\ & = \sum_{s=0}^{\lfloor S/2\rfloor} \frac{(\mu)_{s}}{s!}\sum_{\stackrel{\sum_{i}|\bo m_{i}| = S-2s}{\bm_i\in\Z_+^d}} \frac{H_{\bo m_1}(\bx^1){{...}}H_{\bo m_\ell}(\bx^\ell)H_{\bo m_1}(\by^1){{...}}H_{\bo m_\ell}(\by^\ell)}{\bo m_1!{{...}}\bo m_\ell!},
}
where $(\mu)_s =\frac{\Gamma(\mu+s)}{\Gamma(\mu)}$ is the Pochhammer symbol. In particular, for all $n\geq 1$ we have that
\es{\label{KS-iteration}
& \K_S^{(n-1)}[g](\bx^1,{{...}},\bx^\ell) \\ & = \sum_{s=0}^{\lfloor S/2\rfloor}\bigg(\frac{(\mu)_s}{s!}\bigg)^{n-1}\sum_{\stackrel{\sum_{i}|\bo m_{i}| = S-2s}{\bm_i\in\Z_+^d}} \alpha(\bm_1,{{...}},\bm_\ell)H_{\bo m_1}(\bx^1){{...}}H_{\bo m_\ell}(\bx^\ell),
}
if $g(\bx^1,{{...}},\bx^\ell)$ is a polynomial of the following form
\es{\label{F-special-form}
g(\bx^1,{{...}},\bx^\ell) = \sum_{s=0}^{\lfloor S/2\rfloor}\sum_{\stackrel{\sum_{i}|\bo m_{i}| = S-2s}{\bm_i\in\Z_+^d}} \alpha(\bm_1,{{...}},\bm_\ell)H_{\bo m_1}(\bx^1){{...}}H_{\bo m_\ell}(\bx^\ell).
}

To prove the claim \eqref{KS-alt-form} we use the Mehler kernel \eqref{Mehler-kernel}. Recall that $\la=\sqrt{1/k}$. For any $|\om|<1$ we obtain that
\begin{align*}
&\sum_{S\geq 0} K_S(\bx^1,{{...}},\bx^\ell;\by^1,{{...}},\by^\ell)\om^S = \sum_{M\in\M} \frac{\prod_{i=1}^\ell\prod_{j=1}^k H_{\bm^{i,j}}(\la \bx^i)H_{\bm^{i,j}}(\la \by^i)}{M!}\om^{|M|} \\ & = \bigg(\prod_{i=1}^\ell T_\om(\la\bx^i,\la\by^i)\bigg)^k = (1-\om^2)^{-\mu} \prod_{i=1}^\ell T_\om(\bx^i,\by^i)
\\ & = \bigg(\sum_{a\geq 0} \frac{(\mu)_a}{a!}\om^{2a}\bigg)\bigg(\sum_{b\geq 0}\bigg[\sum_{\stackrel{\sum_{i}|\bo m_{i}| = b}{\bo m_i\in\Z_+^d}}\frac{H_{\bo m_1}(\bx^1)H_{\bo m_1}(\by^1)}{\bo m_1!}{{...}}\frac{H_{\bo m_\ell}(\bx^\ell)H_{\bo m_\ell}(\by^\ell)}{\bo m_\ell!}\bigg]\om^{b}\bigg) \\
& = \sum_{S\geq 0}\bigg\{\sum_{a=0}^{\lfloor S/2\rfloor} \frac{(\mu)_{a}}{a!}\sum_{\stackrel{\sum_{i}|\bo m_{i}| = S-2a}{\bm_i\in\Z_+^d}} \frac{H_{\bo m_1}(\bx^1){{...}}H_{\bo m_\ell}(\bx^\ell)H_{\bo m_1}(\by^1){{...}}H_{\bo m_\ell}(\by^\ell)}{\bo m_1!{{...}}\bo m_\ell!}\bigg\}\om^S.
\end{align*}
The claim follows by from comparing the coefficients of the power series of the first and last expressions.

\noindent {\it Step 4.} Let $M\in\M$ with $|M|=S$. It is easy to see that $H_M(\bx^1,{...},\bx^\ell)$ has an expansion in terms of Hermite polynomials of the following form
$$
H_M(\bx^1,{{...}},\bx^\ell) = \sum_{s=0}^{S}\sum_{\stackrel{\sum_{i}|\bo n_{i}| = S-s}{\bn_i\in\Z_+^d}} \alpha_M(\bn_1,{{...}},\bn_\ell)H_{\bo n_1}(\bx^1){{...}}H_{\bo n_\ell}(\bx^\ell).
$$
However, since $H_M(-\bx^1,{{...}},-\bx^\ell)=(-1)^SH_M(\bx^1,{{...}},\bx^\ell)$ (recall that $H_n(-x)=(-1)^nH_n(x)$), we deduce that $\alpha_M(\bn_1,{{...}},\bn_\ell)=0$ if $s$ is not even, where $S-s=\sum_{i}|\bo n_{i}|$. We conclude that $H_M(\bx^1,{{...}},\bx^\ell)$ has the special form \eqref{F-special-form}, which by \eqref{KS-iteration} implies that $\K_S[H_M] = H_M$ if $\mu=(k-1)d\ell/2=1$. Using identity \eqref{Pn-coeff} we conclude that
$$
P^{(2)}(M,N)=P(M,N)
$$
for all $M,N\in\M$, if $\mu=1$. We deduce that the operator $\PP_S$ (the restriction of $\PP$ to the subspace $\F_S$) is a projection for any $S\geq 0$. Since the spaces $\F_S$ are orthogonal and their span is dense in $\F$, we deduce that $\PP$ is well-defined in the whole $\F$ and it is also a projection.

\noindent {\it Step 5.} It remains to calculate the norm of $\PP_S$ on the space $\F_S$ for $\mu=(k-1)d\ell/2>1$. By the considerations of the previous step, the fact that the any function $H_L$ for $L\in\M$ has the special form \eqref{F-special-form} allow us to apply H\"older's inequality in \eqref{Pn-coeff} and obtain that
\es{\label{Pn-coeff-est}
& |P^{(n)}(M,L)|^2\\ 
& \ll \int_{(\R^d)^\ell} | \K_S^{(n-1)}[H_L](\bx^1,{{...}},\bx^\ell)|^2 \d\ga_d(\bx^1){{...}}\d\ga_d(\bx^\ell) \\
& = \sum_{s=0}^{\lfloor S/2\rfloor}\bigg( \frac{(\mu)_s}{s!}\bigg)^{2n-2}\sum_{\stackrel{\sum_{i}|\bo n_{i}| = S-2s}{\bn_i\in\Z_+^d}} \frac{|\alpha_L(\bn_1,{{...}},\bn_\ell)|^2}{\bn_1!{{...}}\bn_\ell!} \ll \bigg( \frac{(\mu)_{\lfloor S/2\rfloor}}{\lfloor S/2\rfloor!}\bigg)^{2n-2},
}
for any $M,L\in\M$ and all $n\geq 1$, where the implied constants do not depend on $n$. Note that in \eqref{Pn-coeff-est} we used the fact that the map $s\mapsto \frac{(\mu)_s}{s!}$ is increasing for $s>0$ if $\mu>1$. We can now apply Lemma \ref{norm-lemma} to deduce that 
$$
\|\PP_S\|_{\F_S\to\F_S} \leq  \frac{(\mu)_{\lfloor S/2\rfloor}}{\lfloor S/2\rfloor!}.
$$

\noindent {\it Step 6.} Now we show the reverse estimate found in \eqref{Pn-coeff-est} for some matrix $M=L=M_0$. It is simple to see that we can choose $M_0\in\M$ such that
$$
H_{M_0}(\bx^1,{{...}},\bx^\ell) = H_S(\la x_1^1),
$$
where $\bx^i=(x_1^i,{{...}},x_d^i)$ for $i=1,..,\ell$. It is known that
$$
H_n(\la x) = \sum_{a=0}^n \binom n a \la^a(1-\la^2)^{(n-a)/2}H_{n-a}(0)H_a(x), 
$$
for all $n\geq 0$ and that $H_n(0)\neq 0$ if and only if $n$ is even. We deduce that
\est{
\int_{\R} H_n(\la x)\d\ga(x) = (1-\la^2)^{n/2}H_n(0) \neq 0,
}
if $n$ is even and
\est{
\int_{\R} H_{n}(\la x)H_1(x)\d\ga(x) = n\la(1-\la^2)^{(n-1)/2}H_{n-1}(0) \neq 0
}
if $n$ is odd (recall that $\la=\sqrt{1/k}<1$ since $k\geq 2$). We conclude that
$$
H_{M_0}(\bx^1,{{...}},\bx^\ell) = \sum_{s=0}^{S}\sum_{\stackrel{\sum_{i}|\bo n_{i}| = S-2s}{\bn_i\in\Z_+^d}} \alpha_{M_0}(\bn_1,{{...}},\bn_\ell)H_{\bo n_1}(\bx^1){{...}}H_{\bo n_\ell}(\bx^\ell),
$$
where $\alpha_{M_0}(\bf 0,{{...}},\bf 0)\neq 0$ if $S$ is even and $\alpha_{M_0}((1,0,{{...}},0),\bf 0,{{...}},\bf 0)\neq 0$ if $S$ is odd. We can now use \eqref{Pn-coeff} and \eqref{KS-iteration} to deduce that
\est{
P^{(n)}(M_0,M_0) = \sum_{s=0}^{\lfloor S/2\rfloor}\bigg( \frac{(\mu)_s}{s!}\bigg)^{n-1}\sum_{\stackrel{\sum_{i}|\bo n_{i}| = S-2s}{\bn_i\in\Z_+^d}} \frac{|\alpha_{M_0}(\bn_1,{{...}},\bn_\ell)|^2}{\bn_1!{{...}}\bn_\ell!} \gg  \bigg( \frac{(\mu)_{\lfloor S/2\rfloor}}{\lfloor S/2\rfloor!}\bigg)^{n-1}.
}
We can now apply Lemma \ref{norm-lemma} again to finally conclude that
$$
\|\PP_S\|_{\F_S\to\F_S} = \frac{(\mu)_{\lfloor S/2\rfloor}}{\lfloor S/2\rfloor!}.
$$
This finishes the proof.
\end{proof}

\begin{proof}[\bf Proof of Theorm \ref{equiv-thm-laguerre}]
This proof is analogous to the proof of Theorem \ref{equiv-thm-hermite}. The first step is almost identical but now using the following change of variables: $\by=\sqrt{\frac{\pi p}{1+16\pi^2t^2}}\,\bx$ and $\pi s=\arctan(-4\pi t)$. However, in this case we do not need to use an extra step to transform the multiplication factor $e^{\pi im s }$ into $e^{2\pi im s }$, since it naturally emerges from the form of $e^{i\Delta t}\Psi_m(\bx)$ proved in Lemma \ref{SE-flow-lague-herm}. The different thing here is that the factor $\frac{1-4\pi it}{1+4\pi it}$ has no longer a square root on it.
\end{proof}

\begin{proof}[\bf Proof of Theorem \ref{QS-op-thm}]
\noindent {\it Step 1.} Identity \eqref{int-id-G} easily follows from identities \eqref{crucial-id-laguerre} and \eqref{crucial-id-laguerre-2}.

\noindent {\it Step 2.} We now show show the facts about matrix $\Q_S = [Q(a,S-a,c,S-c)]_{a,c=0,{{...}},S}$ in \eqref{matrix-rep-Q}. The fact that $Q(a,b,c,d)>0$ for all integers $a,b,c,d\geq 0$ is shown in Theorem \ref{Q-facts}. Note that if $a+b=c+d=S$, the explicit formula for $Q(a,b,c,d)$ in \eqref{explict-comb-form} can be written in the following form 
$$
Q(a,b,c,d) = \sum_{u=0}^S F_u(a,S-a)F_u(c,S-c)
$$
where
$$
F_u(a,S-a) = \frac{(S-u)!u!}{a!(S-a)!2^{S}}\bigg(\sum_{r=0}^u(-1)^r\binom a r \binom {S-a} {u-r}\bigg)^2.
$$
Thus, if we define $F_S = [F_a(c,S-c)]_{a,c=0,...,S}$ we obtain $\Q_S = F_S^*F_S$. This implies that $\Q_S$ is positive semi-definite. Now we show that $\Q_S$ is doubly stochastic. For any $\nu>-1$, the Laguerre polynomials satisfy the following summation formula
\es{\label{lague-mult-thm}
L_N^{(a\nu+a-1)}(x_1+{{...}}+x_a) = \sum_{n_1+{{...}}+n_a=N} L^{(\nu)}_{n_1}(x_1){{...}}L^{(\nu)}_{n_a}(x_a).
}
We also have that
\es{\label{lague-diff-thm}
L_n^{(\nu+1)}(x) - L_{n-1}^{(\nu+1)}(x) = L_{n}^{{(\nu)}}(x),
}
for all $n\geq 1$. Recalling that we denote $L^{(0)}_n(x)=L_n(x)$ for simplicity, we can use \eqref{lague-mult-thm} to obtain that
\est{
\sum_{a+b=S} Q(a,b,c,d) & = \int_0^\infty \bigg\{\sum_{a+b=S} L_a(x/2) L_b(x/2)\bigg\}L_c(x/2) L_d(x/2)e^{-x}\dx \\ & = \int_0^\infty L^{(1)}_S(x)L_c(x/2) L_d(x/2)e^{-x}\dx \\
& = \int_0^\infty L^{(1)}_S(x)[L_c(x/2) L_d(x/2)-1]e^{-x}\dx + \int_0^\infty L^{(1)}_S(x)e^{-x}\dx \\
& = \int_0^\infty L^{(1)}_S(x)e^{-x}\dx.
}
In the last identity we used that $L_n(0)=1$ for all $n\geq 0$ and that 
$$
\int_0^\infty L^{(1)}_S(x)P(x)xe^{-x}\dx=0,
$$
for any polynomial $P(x)$ of degree at most $S-1$. Next, we can apply \eqref{lague-diff-thm} to conclude by induction that
\est{
\int_0^\infty L^{(1)}_S(x)e^{-x}\dx & = \int_0^\infty \{L_{S-1}^{(1)}(x) + L_{S}(x)\}e^{-x}\dx = \int_0^\infty L_{S-1}^{(1)}(x)e^{-x}\dx\\
& = \int_0^\infty \{L_{S-2}^{(1)}(x) + L_{S-1}(x)\}e^{-x}\dx =  \int_0^\infty L_{S-2}^{(1)}(x)e^{-x}\dx\\
& = \int_0^\infty L_{S-3}^{(1)}(x)e^{-x}\dx = {{...}} = \int_0^\infty L_{0}^{(1)}(x)e^{-x}\dx \\
& = \int_0^\infty e^{-x}\dx = 1.
}
{\it Step 3.} Clearly, since each $\Q_S$ is doubly stochastic, any function $\p\in\G$ with the property that $\p(a,b)=\p(c,d)$ if $a+b=c+d$ is a fixed point of $\Q$, that is, $\Q\p=\p$. Assume now that $\p\in\G$ is a non-zero function such that 
$$
\langle \p,\Q\p\rangle_\G = \|\p\|^2_\G.
$$
We conclude that $\|\Q\p\|^2_\G=\|\p\|^2_\G$. Since the spaces $\G_S$ are mutually orthogonal, $\Q(\G_S)\subset \G_S$ and their span is dense in $\G$, we deduce that 
$$
\sum_{S\geq 0} \|\Q_S(\p_S)\|^2_\G  = \|\Q\p\|^2_\G = \|\p\|_\G^2 = \sum_{S\geq 0} \|\p_S\|^2_\G,
$$
where $\p_S\in\G_S$ is the projection of $\p$ in $\G_S$ (that is, $\p_S(a,b)=\p(a,b)$ if $a+b=S$ and $\p_S(a,b)=0$ otherwise). This implies that $\|\Q_S(\p_S)\|^2_\G=\|\p_S\|^2_\G$ for every $S\geq 0$. Since the matrix \eqref{matrix-rep-Q} that represents $\Q_S$ is positive semi-definite and has strictly positive entries, the Ostrowski's Theorem (see \cite[15.820]{GR}) guarantees that all the eigenvalues of $\Q_S$ other than $1$ are non-negative and strictly less than $1$, hence the vector $(\p(0,S),\p(1,S-1),...,\p(S,0))$ must be a multiple of $(1,1,...,1)$. This finishes the proof.
\end{proof}

\begin{proof}[\bf Proof of Corollary \ref{radial-St-sharp}]
The inequality \eqref{radial-St-ineq} easily follows from Theorem \ref{equiv-thm-laguerre} for $d=2$ and $p=q=4$ and Theorem \ref{QS-op-thm} identities \eqref{int-id-G} and \eqref{int-id-G-2}. Assume that $f(\bx)=\sum_{n\geq 0} \alpha(n)\Psi_n(\bx)$ is a non-zero radial function that maximizes \eqref{radial-St-ineq}. We can use Theorem \ref{QS-op-thm} again to deduce that $\p(a,b)=\alpha(a)\alpha(b)$ belongs to $\G$, is not identically zero and satisfies $\|\Q\p\|_\G=\|\p\|_\G$. We conclude that $\p(a,b)$ depends only on $a+b$. It is then a simple task to verify by induction on $a+b$ that
$$
\p(a,b)=\alpha(0)^2\left(\frac{\alpha(1)}{\alpha(0)}\right)^{a+b}
$$
for all $a,b\geq 0$ and $\alpha(0)\neq 0$. Moreover, since the norm of $\p$ in $\G$ is finite, we must have $\left|\frac{\alpha(1)}{\alpha(0)}\right|<1$. This implies that $\alpha(n)=\alpha(0)\left(\frac{\alpha(1)}{\alpha(0)}\right)^n$ for all $n\geq 0$. We can now use the generating function for the Laguerre polynomials
$$
\frac{e^{-x\om/(1-\om)}}{1-\om} = \sum_{n\geq 0} L_n(x)\om^n
$$
for $\om=\frac{\alpha(1)}{\alpha(0)}$ to deduce that $f(\bx)=\alpha(0)e^{-\pi\frac{1+\om}{1-\om}\|\bx\|^2}$, where $\re \frac{1+\om}{1-\om} = \frac{1-|\om|^2}{|1-\om|^2}>0$. This finishes the proof.
\end{proof}

\subsection{The Spherical Harmonics and Gegenbauer Polynomials Part}

\begin{proof}[\bf Proof of Theorem \ref{radial-id-thm}]
We make use of the Delta Calculus to reduce the problem to a bi-linear form on the sphere $L^2(\Sp^{d-1})$ (we refer to \cite{DF} for a short review of the basics aspects). Let $f\in L^2(\R^d,\d\bx)$ be a smooth function and define $g(r,\bxi) = \ft f(r\bxi)$ for any $r>0$ and $\bxi\in \Sp^{d-1}$. We obtain

\begin{align*}
& \int_\R \int_{\R^d} |e^{it\Delta}f(\bx)|^2 \frac{\d\bx}{\|\bx\|^2}\dt \\ & = \int_\R \int_{\R^d} \bigg(\int_{(\R^d)^2} \ft f(\by) \ov{\ft f(\bz)} e^{-4\pi^2 it(\|\by\|^2-\|\bz\|^2)+2\pi i\bx\cdot(\by-\bz)}\d\by\d\bz\bigg)\frac{\d\bx}{\|\bx\|^2} \dt \\
& = \frac{\Gamma(d/2-1)}{\pi^{d/2-2}}\int_{(\R^d)^2} \ft f(\by) \ov{\ft f(\bz)} {\bo \delta}\big(2\pi(\|\by\|^2-\|\bz\|^2)\big) \frac{\d\by\d\bz}{\|\by-\bz\|^{d-2}} \\
& = \frac{\Gamma(d/2-1)}{2\pi^{d/2-1}}\int_{(\R^d)^2} \ft f(\by) \ov{\ft f(\bz)} {\bo \delta}\big(\|\by\|-\|\bz\|\big) \frac{\d\by\d\bz}{(\|\by\|+\|\bz\|)\|\by-\bz\|^{d-2}} \\
& = \frac{\pi}{(d-2)|\Sp^{d-1}|}\int_0^\infty \int_{\Sp^{d-1}} g(r,\bxi)\bigg\{ \int_{\Sp^{d-1}} \ov{g(r,\bze)} \frac{\d\bze}{\|\bxi-\bze\|^{d-2}} \bigg\} \d\bxi r^{d-1}\d r,
\end{align*}
where in the second identity above we used that $\|\cdot\|^{-2} \mapsto \frac{\Gamma(d/2-1)}{\pi^{d/2-2}}\|\cdot\|^{2-d}$ via the Fourier Transform \eqref{FT}. 

This proves identity \eqref{radial-identi-operat-R}. The above calculation begs us to study the boundedness of the operator $R$ defined in \eqref{radial-operator} which has the following alternative form
\est{
R(g)(\bxi) =  \frac{1}{2^{d/2-1}}\int_{\Sp^{d-1}} {g(\bze)} \frac{\d\bze}{(1-\bxi\cdot\bze)^{d-2}}.
}
We now prove inequality \eqref{R-sharp-ineq}. Any function $g\in L^2(\Sp^{d-1})$ can be decomposed in the following form
$$
g(\bxi) = \sum_{n\geq 0} Y_n(\bxi),
$$
where $Y_n(\bxi)$ is a spherical harmonic of degree $n$. We note that the set $\{Y_n(\bxi)\}_{n\geq 0}$ is always orthogonal in $L^2(\Sp^{d-1})$ and that $Y_0(\bxi) \equiv$ [the mean of $g(\bxi)$ over $\Sp^{d-1}$]. In particular, this implies that
\es{\label{dist-to-const-formula}
{\rm Dist}(g,{\rm Const})^2 = \sum_{n\geq 1} \|Y_n\|^2_{L^2(\Sp^{d-1})}.
}
If $g\equiv Y_n$, we can calculate $R(g)(\bxi)$ by using the Funk-Hecke formula (see \cite[Theorem 1.2.9]{DX})
\es{\label{funk-hecke-formula}
R(Y_n)(\bxi) & = \frac{1}{2^{\nu}}\int_{\Sp^{d-1}} Y_n(\bze) F(\bxi\cdot \bze)\d\bze \\ & = \bigg(\frac{|S^{d-2}|}{2^{\nu}C^{\nu}_n(1)}\int_{-1}^1 C^{\nu}_n(u)F(u)(1-u^2)^{\nu-1/2}\bigg)  Y_n(\bxi),
}
where $F(u)=(1-u)^{-\nu}$ and $\{C^{\nu}_n(u)\}_{n\geq 0}$ are the Gegenbauer polynomials with parameter $\nu=d/2-1$ (in fact, the Funk-Hecke calculation states that the above formula holds for any reasonable $F(u)$). 

For any $\nu>-1/2$, the Gegenbauer polynomials are defined as the orthogonal polynomials with respect to the measure $(1-u^2)^{\nu-1/2}\du$ ($u\in(-1,1)$) and normalized by the condition
\es{\label{C-at-1}
C^\nu_n(1) = \frac{\Gamma(n+2\nu)}{\Gamma(2\nu)n!}.
}
We can use their relation with the Jacobi polynomials \cite[8.962-4]{GR} together with formula \cite[7.391-4]{GR} (for $\rho=\nu-1/2-a, \alpha=\beta=\nu-1/2$) to deduce that
\es{\label{magic-formula-Geg}
&\frac{1}{C^{\nu}_n(1)}\int_{-1}^1 C^{\nu}_n(u)(1-u)^{-a}(1-u^2)^{\nu-1/2} \du\\
& = 2^{2\nu-a}\frac{\Gamma(\nu+1/2)\Gamma(\nu+1/2-a)\Gamma(n+a)}{\Gamma(a)\Gamma(2\nu+n+1-a)}
}
for any $a<1/2+\nu$ and any $\nu>-1/2$. We can now use this formula at $a=\nu=d/2-1$ in conjunction with \eqref{funk-hecke-formula} and \eqref{C-at-1} to deduce that
$$
R(Y_n)(\bxi) = \frac{d-2}{2n+d-2}|\Sp^{d-1}|Y_n(\bxi)
$$
for all $n\geq 0$. We conclude that for $g(\bxi) = \sum_{n\geq 0} Y_n(\bxi)$ we have
\begin{align*}
\int_{\Sp^{d-1}} g(\bxi) \ov{R(g)(\bxi)} \d\bxi & = |\Sp^{d-1}|\sum_{n\geq 0}  \frac{d-2}{2n+d-2}\|Y_n\|^2_{L^2(\Sp^{d-1})} \\
& = |\Sp^{d-1}|\bigg\{\|g\|^2_{L^2(\Sp^{d-1})}  - \sum_{n\geq 1}  \frac{2n}{2n+d-2}\|Y_n\|^2_{L^2(\Sp^{d-1})}\bigg\}\\
& \leq|\Sp^{d-1}|\bigg\{\|g\|^2_{L^2(\Sp^{d-1})}  - \frac{2}{d}\sum_{n\geq 1} \|Y_n\|^2_{L^2(\Sp^{d-1})}\bigg\} \\
& = |\Sp^{d-1}|\bigg\{\|g\|^2_{L^2(\Sp^{d-1})}  - \frac{2}{d}{\rm Dist}(g,{\rm Const})^2\bigg\}.
\end{align*}
This concludes the theorem.  
\end{proof}

\begin{proof}[\bf Proof of Theorem \eqref{radial-thm}]
It is easy to see that Theorem \ref{radial-thm} will follow from Theorem \ref{radial-id-thm} once we prove the following lemma.

\begin{lemma}
Let $f\in L^2(\R^d)$ and write $g(r,\bxi) = f(r\bxi)$ for any $r>0$ and $\bxi\in \Sp^{d-1}$. Then we have
$$
{\rm Dist}(f,{\rm Radial})^2 = \int_0^\infty {\rm Dist}(g(r,\cdot),{\rm Const})^2 r^{d-1}\d r.
$$
\end{lemma}

Let $f\in L^2(\R^d,\dx)$. The functions $g(r,\bxi)=f(r\bxi)$ for $\bxi\in \Sp^{d-1}$ are measurable for almost every $r>0$ and thus we can decompose them in spherical harmonics
$$
g(r,\bxi) = \sum_{n\geq 0} Y_{n,r}(\bxi)
$$
for almost every $r>0$. The projection $Pf(r)$ of $f(\bx)$ in the space of radial functions of $L^2(\R^d,\d\bx)$ is given by
$$
Pf(r) = \frac{1}{|\Sp^{d-1}|}\int_{\Sp^{d-1}} g(r,\bxi) \d\bxi = Y_{0,r}.
$$
By \eqref{dist-to-const-formula} we obtain
$$
\int_{\Sp^{d-1}} |g(r,\bxi) - Pf(r)|^2 \d\bxi = {\rm Dist}(g(r,\cdot),{\rm Const})^2
$$
for almost every $r>0$ and we conclude that
\est{
{\rm Dist}(f,{\rm Radial})^2 & = \int_0^\infty \int_{\Sp^{d-1}} |g(r,\bxi) - Pf(r)|^2 \d\bxi r^{d-1}\d r \\ & = \int_0^\infty  {\rm Dist}(g(r,\cdot),{\rm Const})^2 \d r.
}
This finishes the lemma and the theorem.
\end{proof}

\section{Appendix}\label{appendix}
We now explain the unexpected combinatorial interpretation that we discovered for the coefficients $Q(a,b,c,d)$ in \eqref{Q-coeff}. Let $a,b,c,d$ be non-negative integers and let $N=a+b+c+d>0$. Consider the set $\W(a,b,c,d)=Span\{1^a2^b3^c4^d\}$ of all words formed with $a$ $1$'s, $b$ $2$'s, $c$ $3$'s  and $d$ $4$'s. We can define a way of measuring how different two given words $w=\ell_1\ell_2{{...}}\ell_N$ and $w'=\ell'_1\ell'_2{{...}}\ell'_N$ in $\W(a,b,c,d)$ are by considering the following distance function: $D(w,w')=\#\{i:\ell_i \neq \ell'_i\}$. Now consider the elementary word
$$
e=\underbrace{1 {{...}} 1}_{a} \, \underbrace{2 {{...}} 2}_{b} \, \underbrace{3 {{...}} 3}_{c} \, \underbrace{4 {{...}} 4}_{d}.
$$
We say that a word $w\in \W(a,b,c,d)$ is even if $D(w,e)$ is even, if not we say that $w$ is odd. If we consider the word $w$ to be formed by permuting the letters of the word $e$, then $D(w,e)$ measures how many letters were moved from its original block. For instance, if $a=b=c=d=1$ then $e=1234$ and $w=1243$ is an even word while $w'=1342$ is an odd word. Let $\W_{even}(a,b,c,d)$ and $\W_{odd}(a,b,c,d)$ denote the sets of even and odd words in $\W(a,b,c,d)$ respectively. In what follows we use the convention $\binom n m=0$ if $m<0$ or $m>n$.

\begin{theorem}\label{Q-facts}
For any $a,b,c,d\geq 0$ with $N=a+b+c+d>0$ we have 
\es{\label{comb-connec-form}
Q(a,b,c,d)=\frac{\#\W_{even}(a,b,c,d)-\#\W_{odd}(a,b,c,d)}{2^{N}}.
}
Moreover, we have the following formula
\es{\label{explict-comb-form}
&2^{N}Q(a,b,c,d) \\ & = \sum_{u=0}^U \frac{(a+b-u)!(c+d-u)!(u!)^2}{a!b!c!d!}\bigg(\sum_{r,s=0}^u (-1)^{r+s}\binom a r \binom b {u-r}\binom c s \binom d {u-s}\bigg)^2,
}
where $U=\min\{a+b,c+d\}$. In particular, $Q(a,b,c,d)>0$.
\end{theorem}

\noindent {\bf Remark.} The combinatorial connection \eqref{comb-connec-form} was first proved by Askey, Ismail and Koorwinder in \cite{AIK} using the Master MacMahon Theorem. They also showed that $Q(a,b,c,d)>0$ using analytic tools from the theory of Laguerre polynomials. Formula \eqref{explict-comb-form} was first found by Gillis and Kleeman in \cite{GK} using again analytic tools in conjunction with a combinatorial argument. Later, Gillis and Zeilberger in \cite{GZ} found a pure (and quite clever) combinatorial argument for formula \eqref{explict-comb-form}.

\section*{Acknowledgments}
The author is grateful to Emanuel Carneiro, Diogo Oliveira e Silva and Neal Bez for the helpful comments. The author is also grateful for the math inspiration acquired during the time he spent at The University of Texas at Austin working under the guidance of William Beckner.

\end{document}